\documentclass[11pt,onecolumn,oneside,leqno,final]{amsart}

\usepackage{hyperref}
\hypersetup{
  colorlinks = true,
  urlcolor = blue,
  linkcolor = blue,
  citecolor = red
}
\usepackage{color}

\usepackage{enumerate}

\usepackage{amsfonts}
\usepackage{amsthm, amssymb}
\usepackage[abbrev]{amsrefs}

\newtheorem{thm}{Theorem}[section]
\newtheorem{prop}[thm]{Proposition}
\newtheorem{corl}[thm]{Corollary}
\newtheorem{lemma}[thm]{Lemma}

\newtheorem*{claim*}{Claim}
\newtheorem*{remark}{Remark}

\theoremstyle{definition}
\newtheorem{define}[thm]{Definition}

\newcommand\xra{\xrightarrow}

\newcommand\ten\otimes
\newcommand\isom\simeq

\newcommand\Q{\mathbb{Q}}
\newcommand\Z{\mathbb{Z}}
\newcommand\R{\mathbb{R}}
\newcommand\C{\mathbb{C}}
\newcommand\N{\mathbb{N}}

\newcommand\G{\mathbb{G}}

\newcommand\p{\mathfrak{p}}

\newcommand\calB{\mathcal{B}}
\newcommand\calC{\mathcal{C}}

\newcommand\calH{\mathcal{H}}
\newcommand\calN{\mathcal{N}}
\newcommand\calO{\mathcal{O}}

\newcommand\eps{\varepsilon}

\newcommand\abs[1]{\vert#1\vert}

\newcommand\ang[1]{\langle#1\rangle}

\newcommand\ncv{\vol^{\rm nc}}

\DeclareMathOperator{\vol}{vol}

\DeclareMathOperator{\Min}{Min}
\DeclareMathOperator{\Fix}{Fix}

\DeclareMathOperator{\Isom}{Isom}

\DeclareMathOperator{\rk}{rk}

\DeclareMathOperator{\hrank}{-rank}

\DeclareMathOperator{\GL}{GL}
\DeclareMathOperator{\SO}{SO}
\DeclareMathOperator{\SL}{SL}

\title[On the number of finite subgroups of a lattice]{On the number of finite \\subgroups of a lattice}

\author{Iddo Samet}

\address{University of Illinois at Chicago\\
Department of Mathematics, Statistics, and Computer Science\\
Chicago, IL 60607}
\email{samet@math.uic.edu}

\date{September 2012}

\begin{document}

\begin{abstract}
  We show that the number of conjugacy classes of maximal finite subgroups of a lattice in a semisimple Lie group is linearly bounded by the covolume of the lattice. Moreover, for higher rank groups, we show that this number grows sublinearly with covolume. We obtain similar results for isotropy subgroups in lattices. Geometrically, this yields volume bounds for the number of strata in the natural stratification of a finite-volume locally symmetric orbifold.
\end{abstract}

\maketitle

\section{Introduction}\label{sec:intro}

Several families of infinite groups share the property that they have finitely many conjugacy classes of finite subgroups (henceforth, \emph{the finiteness property}).
In the realm of linear groups this was first proven by Jordan for $\GL_n(\Z)$. Using their reduction theory, Borel and Harish-Chandra generalized Jordan's theorem and proved the finiteness property holds for arithmetic groups of the form $\G(\Z)$, where $\G$ is a linear algebraic group defined over $\Q$ \cite{borel-reduction-theory}. This result was extended by Grunewald and Platonov to general arithmetic groups, as well as to their finite extensions \cite{grunewald-platonov-finite-ext}.
Other families of groups known to enjoy this property are $\mathrm{Aut}(F_n)$, $\mathrm{Out}(F_n)$ (cf.\ \cite{culler-finite-out-fn}), mapping class groups (cf.\ \cite{bridson-finiteness}), word hyperbolic groups, and CAT(0) groups (cf.\ \cite{bridson-haefliger}).

For a family of groups that has the finiteness property, it is natural to seek asymptotic bounds for the number of conjugacy classes of finite subgroups. For example, if $\Gamma$ is a group with the finiteness property, and $\Gamma_n < \Gamma$ is a sequence of finite-index subgroups, then by elementary group theory
\begin{equation*}
  F(\Gamma_n) \leq F(\Gamma) \cdot [\Gamma : \Gamma_n],
\end{equation*}
where $F(\cdot)$ is the number of conjugacy classes of finite subgroups. We generalize this ``linear'' bound to a families of lattices in a given semisimple Lie group. Naturally, the index of a group is replaced by its covolume. For reasons that will be made clear later, we only bound the number of conjugacy classes of \emph{maximal} finite subgroups.

\begin{thm}\label{thm:finite-subgroups}
  Let $G$ be a connected semisimple Lie group with finite center and without compact factors. For a lattice $\Gamma < G$, denote by $f(\Gamma)$ the number conjugacy classes of maximal finite subgroups in $\Gamma$. Then
  \begin{equation*}
    f(\Gamma) \leq c \cdot \vol(G/\Gamma),
  \end{equation*}
  with a constant $c = c(G)$.
\end{thm}

In particular, this theorem establishes the finiteness property for lattices in these semisimple Lie groups. We remark, that if $G$ is simple, mere finiteness already follows from the aforementioned results, unless $\Gamma$ is a non-unform lattice in $SO(d,1)$ ($d \geq 3$) or $SU(d,1)$. Indeed, if $G$ has higher rank, or is $Sp(d,1)$ or $F_4^{-20}$, every lattice is arithmetic. If $G=SO(2,1)$, or $G$ has rank one and $\Gamma$ is a uniform lattice, then $\Gamma$ is word hyperbolic.

\medskip
In some cases, we can make a stronger statement on the asymptotical growth of the number of conjugacy classes of maximal finite subgroups.

\begin{thm}\label{thm:finite-subgroups-high-rank}
  Let $G$ be as in Theorem \ref{thm:finite-subgroups}, and assume moreover, that $\R\!\hrank(G) \geq 2$ and $G$ has Kazhdan's property (T). If $\Gamma_n$ is a sequence of pairwise non-conjugate irreducible lattices in $G$, then
  \begin{equation*}
    \lim_n \frac{f(\Gamma_n)}{\vol(G/\Gamma_n)} = 0.
  \end{equation*}
\end{thm}

The statement of the theorem does not hold for $G = \SO(d,1)$ ($d \geq 2$). Indeed, in section \ref{sec:so}, we exhibit a sequence of lattices $\Gamma_n < G$ such that $\vol(G / \Gamma_n) \to \infty$ and
  \begin{equation*}
    \liminf_n \frac{f(\Gamma_n)}{\vol(G/\Gamma_n)} > 0.
  \end{equation*}
The asymptotical behavior for lattices in other rank-one groups, as well as that for irreducible lattices in products of rank-one groups, remains unsettled.

\medskip

The proof of Theorem \ref{thm:finite-subgroups} is based on an analysis of the action of $\Gamma$ on the associated symmetric space $X = K \backslash G$. Each finite subgroup of $\Gamma$ fixes a connected complete totally geodesic submanifold of $X$. This establishes a one-to-one correspondence between totally geodesic submanifolds fixed by some finite subgroup, and \emph{isotropy} subgroups of $\Gamma$. In this correspondence, maximal finite subgroups correspond to minimal fixed submanifolds. The geometric analogous of Theorem \ref{thm:finite-subgroups} is stated in section \ref{sec:proofs} and proved by techniques of non-positive curvature.

\medskip
From the geometric point of view, we can extend this result. The quotient orbifold $M = X / \Gamma$ has a natural stratification whose strata are the sets
\begin{equation*}
  M_{[H]} = \{  x \in M \; | \; [ \Gamma_{\tilde x} ] = [ H ] \},
\end{equation*}
where $H < \Gamma$ is a finite subgroup, $\Gamma_{\tilde x}$ is the stabilizer in $\Gamma$ of a lift of $x$, and brackets represent conjugacy classes. In this setting, a conjugacy class of a maximal finite subgroup corresponds to a stratum that does not contain any other stratum in its closure. The number of strata equals the number of conjugacy classes of isotropy subgroups in $\Gamma$. In section \ref{sec:isotropy} we prove the following extension of Theorems \ref{thm:finite-subgroups} and \ref{thm:finite-subgroups-high-rank}, which we state here in its geometric form.

\begin{thm}\label{thm:strata}
   Let $X$ be a global symmetric space $X$ of non-compact type. Let $M = X / \Gamma$ be an $X$-orbifold. Denote by $s(M)$ the number of strata in the natural orbifold stratification of $M$. Then
   \begin{equation*}
     s(M) \leq c \cdot \vol(M),
   \end{equation*}
   with a constant $c = c(X)$. If the rank of $X$ is at least 2, and $\Isom(X)$ has property (T), then for any sequence $M_n$ of irreducible $X$-orbifolds that are pairwise non-isometric we have
   \begin{equation*}
     \lim_{n \to \infty} \frac{s(M_n)}{\vol(M_n)} = 0.
   \end{equation*}
\end{thm}

\medskip

Lastly, we remark on the related problem of bounding the \emph{size} of maximal finite subgroups, and the number of torsion elements in $\Gamma$. A bound on the former, combined with a bound on the number of maximal finite subgroups, would yield a bound (perhaps not optimal) on the latter. We were unable to achieve such bounds with the tools used in this work. In fact, in $\SL_2(\R)$, the Gauss-Bonnet formula shows that there are arbitrarily large finite subgroups in lattices of bounded covolume. However, if $\Gamma < G$ is an arithmetic lattice and $H < \Gamma$ is a finite subgroup, then by \cite{agol-belol-finiteness},
\begin{equation*}
  \abs{F} \leq c_1 (\log \vol(G / \Gamma))^{c_2},
\end{equation*}
where $c_1,c_2$ are constants that depend on $G$ alone\footnote{This follows by combining the inequalities of 4.3, 4.5, and 4.6 of \cite{agol-belol-finiteness}. These results are proved for standard lattices in $\SO(n,1)$. In a private communication, Mikhail Belolipetsky reassured me that this is true in full generality, i.e. for arithmetic lattices in a semisimple algebraic group.}.

Unfortunately, this does not prove that the number of torsion elements in a lattice in a higher rank simple Lie group grows sublinearly with volume, a statement that which we conjecture to be true.

\subsection*{Acknowledgments}
Preliminary results of this work were achieved during the author's Ph.D. work at the Hebrew University of Jerusalem. I would like thank my advisor, Tsachik Gelander, for his continuous support, and hies encouragement to extend and generalize the scope of these results. I would also like to thank Misha Belolipetsky, Martin Bridson, Alex Furman, and Alex Lubotzky for stimulating conversations.

\section{Generalities}\label{sec:general}

We review some facts regarding isometries of Riemannian globally symmetric spaces. Let $X$ be a Riemannian globally symmetric space of non-positive sectional curvature. We will always assume the metric is rescaled so that $-1 \leq K \leq 0$.

\subsection{}
The isometry group $\Isom(X)$ is a Lie group without center, with finitely many connected components. Its identity component can be realized as the connected component of the real points of a linear algebraic group.

For $g \in \Isom(X)$, we define the displacement function
\begin{equation*}
  d_g : X \to \R^{\geq 0}, \; \; d_g(x) = d(gx,x).
\end{equation*}
The minimal displacement set is defined
\begin{equation*}
  \Min(g) = \{ x \in X \; : \; d_g(x) = \inf d_g \}.
\end{equation*}
An element $g \in \Isom(X)$ is \emph{semisimple} if $\Min(g)$ is non-empty. If a semisimple element fixes a point in $X$ it is \emph{elliptic}. Otherwise, it is \emph{hyperbolic}. In either case, $\Min(g)$ is a connected complete totally geodesic submanifold, and is thus a globally symmetric space \cite{helgason-book}*{IV, \S 7}.
If $\Min(g)$ is empty, $g$ is \emph{parabolic}.
If $g$ is a hyperbolic isometry and $d_g$ is constant, then $g$ is called a Clifford isometry.
We note that $g \in \Isom(X)^\circ$ is a semisimple isometry if and only if it is a semisimple element in the linear algebraic sense. If $g \neq 1$ is a unipotent element in the linear algebraic sense then it is parabolic (for the converse to be true, it is necessary that $\inf d_g =0$).

If $A$ is a set of commuting semisimple isometries then $\bigcap_{\alpha \in A} \Min(\alpha)$ is non-empty.

We remark that if $\Gamma$ is a group acting by isometries on $X$, we will call elements of $\Gamma$ hyperbolic, elliptic or parabolic, according to the classification of the isometry by which they act.

\subsection{}
The displacement function $d_g$ is convex, in the sense that $t \mapsto d_g(c (t))$ is a convex function for every geodesic $c$. Consequentially, the sub-level sets of $d_g$ are convex sets.

If $C$ is a convex subset of $X$ and $x \in X$, there is a unique point $\pi_C(x) \in C$ --- the projection of $x$ to $C$ --- which is closest to $x$. If $C$ is invariant under an isometry $g$, then $d_g(\pi_C(x)) \leq d_g(x)$. This follows from the fact that the projection does not increase distances. Consequentially, if $x \not \in C$ and $c : [0,\infty) \to X$ is a geodesic ray with $c(0) = \pi_C(x)$ and $c(t_0) = x$, then $d_g(c(t))$ is non-decreasing.

\subsection{}
Let $\Gamma$ be a countable group acting properly by isometries on $X$. By \cite{wolf-bounded-isometries}, a Clifford isometry acts by translation on the Euclidean factor of $X$, and trivially on the complementary factor. From this, it follows that the set $T$ of Clifford isometries in $\Gamma$ form a normal finitely generated free abelian subgroup. If $k = \rk_\R T$ then there is an isometric splitting $X \isom X_1 \times X_2$, with $X_2 \isom \R^k$, such that every element of $T$ acts trivially on $X_1$ and as a translation on $X_2$, and that the action of $T$ on $X_2$ is cocompact. We shall call this decomposition the \emph{Clifford splitting}\footnote{This splitting can be regarded as a special case of Gromov's $\text{ess}^a\text{-vol}$ decomposition, with $a=\infty$, cf. \cite{bgs-book}*{\S 12}.} of $X$ with respect to $\Gamma$. Note that this need not coincide with the de Rham decomposition, because $X_1$ may contain a Euclidean factor.

Every element $\gamma \in \Gamma$ preserves the Clifford splitting, and can be written as $(\gamma_1,\gamma_2)$ with $\gamma_i \in \Isom(X_i)$. We denote by $\Gamma_i$ the projection of $\Gamma$ to $\Isom(X_i)$, $i=1,2$.

\begin{prop}\label{prop:proper-on-non-euclid}
  $\Gamma_1$ acts properly on $X_1$.
\end{prop}
\begin{proof}
Suppose in contradiction that $\gamma^{(n)}_1 \in \Gamma_1$ is a sequence such that $\gamma^{(n)}_1 \to \infty$ but $\gamma^{(n)}_1 x_1$ converges for some $x_1 \in X_1$.

There are elements $\gamma^{(n)}_2 \in \Gamma_2$ such that $\gamma^{(n)} = (\gamma^{(n)}_1,\gamma^{(n)}_2) \in \Gamma$. Fix any $x_2 \in X_2$. Since $T$ acts trivially on $X_1$ and cocompactly on $X_2$, by multiplying the $\gamma^{(n)}$ by elements of $T$, we may assume $d(\gamma^{(n)}_2 x_2, x_2) < K$, where $K$ is some constant that depends on $T$.
Passing to a subsequence we have that $\gamma^{(n)}_2 x_2$ converges, and thus $\gamma^{(n)} (x_1,x_2)$ converges. This contradicts to properness of the action of $\Gamma$.
\end{proof}

The kernel of the projection $p_1$ projects to a subgroup $\Gamma_2'$ of $\Gamma_2$ that acts properly on $X_2$. Note that $T$ injects into $\Gamma_2'$, and thus $\Gamma_2'$ acts cocompactly on $X_2$.

As we have remarked, the Euclidean factor of the Clifford decomposition may be smaller than the Euclidean factor of the de Rham decomposition. However, if $X / \Gamma$ has finite volume, then the decompositions coincide \cite{eberlein-symm-difeo}*{Thm I}. In this case, the statement of Proposition \ref{prop:proper-on-non-euclid} follows from \cite{eberlein-symm-difeo}*{Cor.~F}.

\subsection{}
For a group of isometries $\Gamma$, $x \in X$, and $\eps > 0$, we denote
\begin{equation*}
  \Gamma_\eps(x) = \langle \gamma \in \Gamma : d_\gamma(x) < \eps \rangle.
\end{equation*}
Recall the classical Margulis Lemma:

\begin{thm}\cite{thurston-book}*{Ch. 4} \label{lem:margulis}
Let $X$ be a globally symmetric space of non-positive curvature.
There are constants $\eps > 0$ and $m \in \N$ (depending on $X$) such that for every discrete group $\Gamma < \Isom(X)$ and every $x \in X$, $\Gamma_\eps(x)$ contains a normal nilpotent subgroup $N$ of index at most $m$. Moreover, $N$ is the intersection of $\Gamma_\eps(x)$ with a connected nilpotent group in $\Isom(X)$.
\end{thm}

We will need a slightly stronger version of this theorem. Namely, we want the constants $\eps$ and $m$ to depend only on the \emph{dimension} of $X$ (with our standing assumption on curvature bounds). To this end, we first observe that if the metric is rescaled by some factor, then the constant $\eps$ can be rescaled by the same factor, and $m$ is can be left unaltered. Therefore, if the constants of the Margulis Lemma are given for a space with minimal curvature exactly -1, then they will be adequate for any rescaling of this space with curvature $-1 \leq K \leq 0$.

Next, recall that a complete simply connected symmetric space of non-positive curvature admits a decomposition into a product of a Euclidean factor and irreducible symmetric spaces of non-compact type. By the classification of \'E. Cartan, there are --- up to isometry and rescaling of the metric --- only finitely many symmetric spaces of non-compact type of a given dimension \cite{helgason-book}*{Ch. X}.

Therefore, there are finitely many spaces to consider of each dimension, if we rescale the metric such that minimal curvature is -1. It follows that for $d \in \N$, we may choose $\eps$ and $m$ such that the statement of the theorem holds with these constants for every space of dimension $\leq d$.

\medskip
We state two important consequences of the fact that the nilpotent subgroup $N$ stipulated in Theorem \ref{lem:margulis} is contained in a connected nilpotent linear group.
By a theorem of Lie, a connected linear solvable Lie group can be conjugated to a subgroup of upper triangular matrices. Hence, the commutator subgroup of such group consists of unipotent elements. We claim that if $N$ contains a parabolic element then it contains a central parabolic element. If $N$ is abelian there is nothing to prove. Otherwise, the elements of the $k-1$-iterated commutator, where $k$ is the nilpotency rank of $N$, are central, and are parabolic. Similar considerations are made in \cite{gelander-vol-vs-rank}.

Moreover, the semisimple elements in $N$ consist of an abelian subgroup \cite{borel-algebraic-groups}*{\S 10}. In particular, if $H$ is a finite subgroup of $\Isom(X)$, then it has a normal abelian subgroup of index no more than $m$.

\subsection{}
If $\gamma$ is an element in a discrete group for which the Margulis Lemma holds, then $\gamma^i$ is in the corresponding normal nilpotent subgroup, for some $i \leq m$. With the purpose of relating properties of $\gamma^i$ to those of $\gamma$, we make the following definition.

\begin{define}
  A semisimple isometry $g \in \Isom(X)$ is \emph{stable} if $C(g^i)=C(g)$ for every $i = 1,\ldots, m$, where $m = m(X)$ is the constant of the Margulis Lemma.
  Here, $C(g)$ is the centralizer of $g$ in $\Isom(X)$.
\end{define}

We remark that dependence on the constant of the Margulis Lemma introduces ambiguity to this definition. However, since it is an auxiliary notion, we will tolerate this. Moreover, we will later fix the space $X$, and consider subspaces of it. The constant of the Margulis Lemma will be set once and for all, and stability of isometries of a subspace $Y$ of $X$ will be defined using this constant.

Since $C(g)$ acts transitively on $\Min(g)$, stability of $g$ also implies that $\Min(g^i) = \Min(g)$ for $i=1,\ldots,m$.
A slightly weaker notion of stability was introduced by Gromov \cite{bgs-book}*{\S 12}; in his definition, the latter equality is the defining property.

It is easy to see that translations on a Euclidean space are stable (in fact, if $g \in \Isom(\R^n)$ is a translation then $C(g^i) = C(g)$ for every $i \geq 1$). This implies that, in the general setting, Clifford isometries are stable. Indeed, an isometry commutes with a Clifford isometry if and only if their actions on the Euclidean factor commute.

For other semisimple isometries, we have the following ``stabilization'' lemma, proved in \cite{gelander-vol-vs-rank}*{Lemma 1.4} (cf.\ \cite{bgs-book}*{\S 12.5} and \cite{samet-betti}*{Prop. 2.5}).
\begin{lemma}\label{lem:stablizing}
  There is a constant $M = M(X)$ such that for every $g \in Isom(X)$ there exists $i \leq M$ such that $g^i$ is stable.
\end{lemma}
As we have observed, for a given $d \in \N$, it is possible to choose the constants of the Margulis Lemma so that they will be adequate for all symmetric spaces of dimension $\leq d$. The same reasoning shows that in the previous lemma, it is possible to choose a constant $M$ that will suit all spaces of dimension $\leq d$.

\begin{lemma}\label{lem:stable-clifford-commute}
  Let $\alpha \in \Gamma$ be a stable isometry. If $\min d_\alpha < \eps$ (the constant of the Margulis Lemma) then $\alpha$ commutes with every Clifford isometry in $\Gamma$.
\end{lemma}
\begin{proof}
  Let $X = X_1 \times X_2$ be the Clifford splitting with respect to $\Gamma$. Let $R_\lambda : X \to X$ be the map that acts trivially on $X_1$, and by the homothety $x \mapsto \lambda x$ on $X_2$. Note that $R_\lambda$ conjugates $\Gamma$ to a group of isometries of $X$. Moreover, if $\beta$ is a Clifford translation then the displacement of $R_\lambda \beta R_\lambda^{-1}$ is $\lambda$ times the displacement of $\beta$.

  Let $\beta \in \Gamma$ be a Clifford isometry. By conjugating by $R_\lambda$ for $\lambda > 0$ sufficiently small, we may assume that $d_\beta(x) < \eps$. Let $N$ be the normal nilpotent subgroup of $\Gamma_\eps(x)$, with index $i \leq m$. Since $\alpha^i, \beta^i \in N$ are semisimple, they commute. Thus $\beta^i \in C(\alpha^i) = C(\alpha)$, by stability of $\alpha$, and therefore $\alpha \in C(\beta^i) = C(\beta)$ by stability of $\beta$.
\end{proof}

\subsection{}
Let $\Gamma$ be a group acting properly (but possibly not faithfully) on $X$. For $x \in X$, we denote $d_\Gamma(x) = \inf d_\gamma(x)$, where $\gamma$ ranges over all elements of $\Gamma$ that act non-trivially. The $\delta$-thick part of $X / \Gamma$ ($\delta > 0$) is defined as
\begin{equation*}
(X / \Gamma)_{\geq \delta} = \{ x \in X/\Gamma \; : \; d_\Gamma(\tilde x) \geq \delta \}
\end{equation*}
where $\tilde x$ is any lift of $x$ to $X$. We denote $\vol_{\geq \delta}(X / \Gamma) = \vol( (X / \Gamma)_{\geq \delta})$. If $X$ is a point, we set $\vol(X)=1$ and $\vol_{\geq \delta}(X / \Gamma)=1$ for every $\delta > 0$.

In the presence of a Euclidean factor, one cannot hope to relate any property of $\Gamma$ to the volume of $X / \Gamma$; indeed, by conjugating $\Gamma$ by a homothety of the Euclidean factor, the volume of the quotient can be made arbitrarily small. The following definition aims at overcoming this difficulty.

\begin{define}
  Let $X \isom X_1 \times X_2$, $X_2 \isom \R^k$, be the Clifford decomposition with respect to $\Gamma$.
  We define $\ncv(X / \Gamma) = \vol(X_1 / \Gamma_1)$ (for ``no Clifford volume''), and $\ncv_{\geq \delta}(X / \Gamma) = \vol_{\geq \delta}(X_1 / \Gamma_1)$.
\end{define}

The following proposition lets us estimate the volume of the thick part of a space by a discrete set.
\begin{prop}\label{prop:discrete-vol-est}
Let $Y$ be a totally geodesic submanifold of $X$, and let $\Gamma$ be a group acting properly on $Y$. There exists a constant $c = c(X, \delta)$ such that the following holds.
If $\calN$ is a $\delta$-discrete set in $(Y / \Gamma)_{\geq \delta}$ (i.e. every two points in $\calN$ are at distance $\geq \delta$) then
\begin{equation*}
  \abs{\calN} \leq c \cdot \vol_{\geq \frac{\delta}{2}}(Y / \Gamma).
\end{equation*}
Furthermore, if $\calN$ is a maximal $\delta$-discrete set then
\begin{equation*}
  \abs{\calN} \geq c^{-1} \cdot \vol_{\geq \delta}(Y / \Gamma).
\end{equation*}
\end{prop}
The proof is standard, and uses two ideas. First, that balls of radius $\delta/4$ around points of $\calN$ are disjoint in $(Y / \Gamma)_{\geq \delta/2}$, and are injective images of balls in $Y$. Second, that the volume of injected balls can be bounded from below by the volume of Euclidean balls with the same radius. Similarly, if $\calN$ is maximal then the (metric) balls of radius $2 \delta$ around points of $\calN$ cover $(Y / \Gamma)_{\geq \delta}$, and the volume of these balls is can be bounded from above by the Bishop-Gromov theorem.

\section{Fixed submanfiolds}\label{sec:max-finite}

In this section, $X$ is a Riemannian globally symmetric space of non-positive sectional curvature, with $-1 \leq K \leq 0$. The constants $\eps$ and $m$ are the constant of the Margulis Lemma.

\subsection{Submanifolds fixed by maximal finite subgroups}
By a theorem of Kazhdan and Margulis \cite{kazhdan-margulis}, a finite volume locally symmetric manifold of non-compact type has a point with injectivity radius greater than some constant that depends only on the universal cover. We develop a variation on this result.

Let $M = X / \Gamma$ be an orbifold, and let $H < \Gamma$ is a finite group.
If we restrict ourselves to points on $\Fix(H)$, it does not make any sense to seek a point with $d_\Gamma > 0$. The best one can hope for is that the only elements that translate by less than some constant are those of $H$. Of course, even this cannot be expected if there happens to be a hyperbolic element that acts by a small translation on $\Fix(H)$.  Assuming this does not happen, we have to following.

\begin{prop}\label{prop:mini-km}
  Assume $\Gamma$ a group acting properly by isometries on $X$, such that $\vol_{\geq \mu}(X/\Gamma) < \infty$ for every $\mu > 0$.
  Let $\delta < \eps$.
  Then for every maximal finite subgroup $H < \Gamma$, either there is a hyperbolic element $\gamma \in \Gamma$ with $\Min(\gamma) \supseteq \Fix(H)$ and $\min d_\gamma \leq \delta$, or there exists a point $x \in \Fix(H)$ such that $\Gamma_{\delta/2m}(x) = H$.
\end{prop}

Several lemmata will proceed the proof of this proposition. In what follows we denote $S = \Fix(H)$, and by $\Gamma_S$ --- the stabilizer of $S$ in $\Gamma$.

\begin{lemma}\label{lem:max-fixed-embed}
  $S / \Gamma_S$ is a manifold and the map $S / \Gamma_S \to X / \Gamma$ is an embedding.
\end{lemma}
\begin{proof}
  First, we show that every elliptic element in $\Gamma_S$ acts trivially on $S$, proving that $S / \Gamma_S$ is a manifold. If $\gamma \in \Gamma_S$ is elliptic then it fixes some point $x \in S$. The stabilizer of $x$ is a finite group containing $H$ and $\gamma$, hence by maximality $\gamma \in H$. Therefore, $\gamma$ acts trivially on $S$.

  We prove that the map $S / \Gamma_S \to X / \Gamma$ is an embedding. To prove that it injective, assume $x, \gamma x \in S$ for $\gamma \in \Gamma$. On one hand, $H \leq \Gamma_{\gamma x}$ because $\gamma x \in S$. On the other hand, $\gamma H \gamma^{-1} \leq \gamma \Gamma_x \gamma^{-1} = \Gamma_{\gamma x}$. Since $H$ is maximal, there is equality in both cases. Thus $\gamma$ normalizes $H$, and it keeps $S = \Fix(H)$ invariant, that is, $\gamma \in \Gamma_S$.

  To complete the proof, let $x_n$ be a sequence in $S$ and assume that $\gamma_n x_n$ converges to $y \in S$ for some sequence $\gamma_n \in \Gamma$. We show that $\gamma_n \in \Gamma_S$ for sufficiently large $n$.

  Fix $\delta > 0$ such that $\Gamma_\delta(y) = H$ (it exists because $y \in S$, and $\Gamma$ acts properly). For sufficiently large $n$, $d(\gamma_n x_n, y) < \delta / 2$, hence $\gamma_n H \gamma_n^{-1} \subseteq \Gamma_{\delta}(y) = H$, and maximality implies equality. Then $\gamma_n$ normalizes $H$, and hence stablizes $S$.
\end{proof}

We will need the following elementary lemma from \cite{bgs-book}*{\S 12}:
\begin{lemma}\label{lem:inj-point-near-small-group}
  For every $\delta > 0$ and $k \in \N$ there exists $0 < \mu < \delta$, such that following holds. If $\Gamma_\delta(x)$, $x \in X$, is finite of size at most $k$, then there exists a point $y \in X$ such that $d(x,y) < \delta/4$ and $d_\Gamma(y) > \mu$.
\end{lemma}
Here, again, the constant $\mu$ can be chosen to accommodate for all spaces of dimension $\leq \dim(X)$.

\smallskip
For a point $x \in X$, let $\rho(x) = \inf_\gamma d_\gamma(x)$ where  $\gamma$ ranges over all non-elliptic elements in $\Gamma$. Note the if $x \in S$ then $\rho(x)$ is at most $d_{\Gamma_S}(x)$ (w.r.t. the action of $\Gamma_S$ on $S$).

To study the behavior of $\rho$ at points of $S$, it is useful to introduce the notion of quasi-thickness which was defined and studied in \cite{samet-betti}.

\begin{define}
Let $M = X / \Gamma$ be an orbifold.
For $\delta > 0$ and $k \in \N$, the \emph{$(\delta,k)$-quasi-thick part} of $M$ is defined
\begin{equation*}
 M_{\geq \delta, k} = \{ x \in X : \abs{\Gamma_\delta(x)} \leq k \} / \Gamma.
\end{equation*}
\end{define}

We will need the following simple fact.

\begin{claim*}
  If $\vol_{\geq \mu}(M) < \infty$ for every $\mu>0$, then $M_{\geq \delta, k}$ is compact for every $\delta>0, k \in \N$.
\end{claim*}
\begin{proof}
  Let $\delta > 0$ and $k \in \N$. Let $x \in M_{\geq \delta, k}$ and fix a lift $\tilde x \in X$. By Lemma \ref{lem:inj-point-near-small-group}, there is some $\mu < \delta$ such that the ball of radius $\delta / 2$ around $\tilde x$ contains a ball of radius $\mu / 4$ that injects to the $\frac{\mu}{2}$-thick part of $M$ (the displacement function is 2-Lipschitz).

  If $M_{\geq \delta, k}$ is not compact, then there is a sequence of points $x_i$ in $M_{\geq \delta, k}$ such that $d(x_i, x_j) > \delta$ for $i \neq j$. Then the balls of radius $\delta/2$ around these points are disjoint, and the intersection of each of these balls with $M_{\geq \frac{\mu}{2}}$ contains an injected ball of radius $\mu/4$. Since these disjoint injected balls all have the same volume, this contradicts the assumption that $M_{\geq \frac{\mu}{2}}$ has finite volume.
\end{proof}

\begin{lemma}\label{lem:rho-compact}
  For every $a >0$, the image of the set $C = \{ x \in S \; | \; \rho(x) \geq a \}$ in $X / \Gamma$ is compact.
\end{lemma}
\begin{proof}
 Let $x \in X / \Gamma$ be a point with preimage $\tilde x \in C$. Let $N$ be the normal nilpotent subgroup of $\Gamma_{a/2m}(\tilde x)$. Since its index is bounded by $m$, it can be generated in words of length at most $2m$ in the generators of $\Gamma_{a/2m}(\tilde x)$\footnote{For a proof of this, see the proof of Lemma 2.2 in \cite{samet-betti}.}. Hence, $N$ is generated by elements $\gamma$ with $d_\gamma(x) < a$. Since $\rho(\tilde x) \geq a$, these generators must be elliptic, and therefore they commute. It follows that $N$, and hence $\Gamma_{a/2m}(\tilde x)$ is finite. Therefore, $\Gamma_{a/2m}(\tilde x) = H$ by the maximality of the latter. It follows that $x$ is in the $(\frac{a}{2m}, \abs{H})$-quasi-thick part of $M$, which is compact.

 Since $C$ is closed and $\Gamma_S$-invariant, its image in $S / \Gamma_S$ is closed. Hence, by Lemma \ref{lem:max-fixed-embed} its image in $X / \Gamma$ is closed, and the claim follows.
\end{proof}

\begin{proof}[Proof of Proposition \ref{prop:mini-km}]
  Suppose that there are no hyperbolic elements $\gamma \in \Gamma$ with $\min d_\gamma \leq \delta$ and $\Min(\gamma) \supseteq S$.

  It follows from Lemma \ref{lem:rho-compact} that there exists a point $y \in S$ where $\rho$ attains its maximum on $S$. Let us denote by $\#\rho(x)$ the number of elements $\gamma \in \Gamma$ for which $d_\gamma(x) = \rho(x)$. We can assume $y$ is chosen such that $\#\rho(y)$ is minimal (among points in $S$ where $\rho$ is maximal).

  We claim that $\rho(y) \geq \delta$.
  Let us first show that this will prove the proposition. Indeed, we have observed in the proof of Lemma \ref{lem:rho-compact} that if $\rho(y) \geq \delta$ then $\Gamma_{\delta/2m}(y)$ is finite, hence equal to $H$ by maximality.

  Now suppose contrarily that $\rho(y) < \delta$. Let $\mu = \rho(y)$, and let $\Sigma$ be the set of non-elliptic elements $\gamma$ with $d_\gamma(y) = \mu$. Let $\Delta$ be the group generated by $H$ and $\Sigma$, and let $N$ the normal nilpotent subgroup of $\Delta$.

  There are two cases. First, let us assume that $\Sigma$ contains a parabolic element.
  Then $N$ contains parabolic elements, and therefore contains a central parabolic element $\gamma_1 \in N$. Since $N$ has finite index in $\Delta$, $\gamma_1$ has finitely many conjugates, say, $\gamma_1,\ldots,\gamma_k$. Define $D(x) = \sum_i d_{\gamma_i}(x)$. Then for every $a>0$, $C = D^{-1}([0,a])$ is closed, convex, and $\Delta$-invariant. For sufficiently small $a$, $y \not\in C$. Let $z \in C$ be the projection of $y$ onto $C$, and let $c : [0, \infty)$ be a geodesic ray with $c(0) = z$ and $c(t_0) = y$. Every $\gamma \in \Delta$ keeps $C$ invariant, hence $d_\gamma$ is non-decreasing along $c$.
  Also, since $H$ fixes $y$ and keeps $C$ invariant, it also fixes $z$, and hence fixes $c$ pointwise. We claim that for some $\gamma \in \Sigma$, $d_\gamma$ is strictly increasing along $c$. If $d_\gamma \circ c$ is not strictly increasing, then it is locally constant at some point, and since $d_\gamma \circ c$ is analytic, this implies that $d_\gamma$ is constant along $c$. Since this is true for the generators of $\Delta$, $d_\gamma \circ c$ is bounded and thus constant for every $\gamma \in \Delta$. This would imply that $D$ is constant along $c$ and thus $y \in C$, which is a contradiction.

  We have shown that for $t > t_0$, $d_\gamma(c(t)) \geq \mu$ for all $\gamma \in \Sigma$, and $d_\gamma(c(t)) > \mu$ for at least one $\gamma \in \Sigma$. We may choose $t_1 > t_0$ sufficiently small, such that for every $\gamma \in \Gamma$, $d_\gamma(c(t_1)) > \mu$ whenever $d_\gamma(c(t_0)) > \mu$.

  Therefore, $\rho(c(t_1)) \geq \mu$, and if $\rho(c(t_1)) = \mu$ then $\#\rho(c(t_1)) \leq \#\rho(y) - 1$.
  This is a contradiction to the choice of $y$.

  Now, suppose $\Sigma$ does not contain a parabolic element, hence it consists of hyperbolic elements. Let $i$ be the index of $N$ in $\Delta$. Then the elements $\gamma^i$, $\gamma \in \Sigma$, are hyperbolic and thus commute. It follows that $C = \bigcap_{\gamma \in \Delta} \Min(\gamma^i)$ is non-empty, and it is closed, convex and $\Delta$-invariant. Let $z \in C$ be the projection of $y$ to $C$. Then, as before, $z \in S$.
  By our assumption, there are no hyperbolic elements $\gamma$ with $\Min(\gamma) \supseteq S$ and $\min d_\gamma \leq \delta$. Therefore, $C \cap S$ is properly contained in $S$. As in the previous case, we take a geodesic ray in $S$ which starts at $z$ and is perpendicular to $C$ at $z$. Along this geodesic, $d_\gamma$ is non-decreasing for every $\gamma \in \Delta$, and is strictly increasing for at least one $\gamma \in \Sigma$. This leads to contradiction, as in the previous case.
\end{proof}

\subsection{Stable singular submanifolds}
We now formulate a key proposition which allows inductive reasoning needed for the proof of the main theorem. The statement of the proposition is an analogy to Gromov's \cite{bgs-book}*{Thm 12.11}, and parts of its proof follow along the lines of his proof.

Let $A \subset \Gamma$ be a set of stable semisimple isometries. If $Y = \bigcap_{\alpha \in A} \Min(\alpha)$ is non-empty then it is a connected complete totally geodesic submanifold of $X$. If, furthermore, $\min d_\alpha < \eps$ for every $\alpha \in A$, we call $Y$ a \emph{stable singular submanifold}.
Two submanifolds $Y_1, Y_2$ are called \emph{non-conjugate} if $Y_2 \neq \gamma Y_1$ for all $\gamma \in \Gamma$.

Note that if $A$ defines a stable singular submanifold $Y$ then the elements of $A$ commute. Indeed, if $\alpha_1, \alpha_2 \in A$ then $\alpha_1,\alpha_2 \in \Gamma_\eps(y)$ for some $y \in Y$. If $N$ is the normal nilpotent subgroup of $\Gamma_\eps(y)$, then $\alpha_1^i,\alpha_2^i \in N$ for some $i \leq m$. It follows that $\alpha_1^i$ and $\alpha_2^i$, hence by stability also $\alpha_1$ and $\alpha_2$, commute.

For a submanifold $Y$ we denote by $\Gamma_Y$ the stabilizer of $Y$ in $\Gamma$.

\begin{prop}\label{prop:inductive-vol-bound}
  Let $W$ be a connected complete totally geodesic submanifold of $X$, and let $\Gamma$ be a group acting properly by isometries on $W$. Let $\Sigma$ be a set of non-conjugate stable singular proper submanifolds in $W$. For every $\eps_1>0$ there is are constants $\eps_2 = \eps_2(\eps_1,X)$ and $c = c(\eps_1,X)$ such that
  \begin{equation}\label{eq:bound-for-stable-ess-vol}
    \sum_{Y \in \Sigma} \ncv_{\geq \eps_1}(Y / \Gamma_Y) \leq c \cdot \ncv_{\geq \eps_2}(W/\Gamma).
  \end{equation}
\end{prop}

\begin{proof}
The proof is by induction on the dimension of $W$.
Clearly, the claim is true for submanifolds of dimension 0 because there are no proper submanifolds.

For $Y \in \Sigma$, we denote by $\Sigma_Y = \{ Z \in \Sigma \; : \; Z \subsetneq Y \}$. Let $Z \in \Sigma_Y$ and denote by $A_Y, A_Z$ the (maximal) set of stable elements such that $Y = \bigcap_{\alpha \in A_Y} \Min(\alpha)$ and $Z = \bigcap_{\alpha \in A_Z} \Min(\alpha)$. Then $A_Y \subsetneq A_Z$. Fix some $z \in Z$. The elements of $A_Z$ commute, from which we deduce that every element of $A_Z$ keeps $Y$ invariant. Thus, $Z$ can be considered a stable submanifold of $Y$ with respect to $\Gamma_Y$.

By induction, we may assume there are constants $\eps_2'$ and $c'$ such that \eqref{eq:bound-for-stable-ess-vol} holds for submanifolds with dimension lower than $\dim(W)$. Hence, considering the action of $\Gamma_Y$ on $Y$,
\begin{equation*}
  \sum_{Z \in \Sigma_Y} \ncv_{\geq \eps_1}(Z / \Gamma_{Z,Y}) \leq c' \cdot \ncv_{\geq \eps_2'}(Y / \Gamma_Y),
\end{equation*}
where $\Gamma_{Z,Y} = \Gamma_Y \cap \Gamma_Z$.

Note that since $\Gamma_{Z,Y} < \Gamma_Z$, it is not hard to deduce that $\ncv_{\geq \eps_1}(Z / \Gamma_Z) \leq \ncv_{\geq \eps_1}(Z / \Gamma_{Z,Y})$.
The important point is that every element of $A_Y$ keeps $Z$ invariant, so it commutes with every Clifford isometry in $\Gamma_Z$ (by Lemma \ref{lem:stable-clifford-commute}). Therefore, every Clifford isometry in $\Gamma_Z$ keeps $Y$ invariant. It follows that $Z$ has the same Clifford splitting w.r.t $\Gamma_Z$ or $\Gamma_{Z,Y}$.

Let $\Sigma' \subseteq \Sigma$ be the set of stable submanifolds which are maximal in $\Sigma$ with respect to inclusion. Then we have
\begin{equation*}
\begin{split}
   &\sum_{Y \in \Sigma} \ncv_{\geq \eps_1}(Y / \Gamma_Y) \leq
   \sum_{Z \in \Sigma'} \sum_{Z \supseteq Y \in \Sigma} \ncv_{\geq \eps_1}(Y / \Gamma_Y) \\
   &\leq \sum_{Z \in \Sigma'} \ncv_{\geq \eps_1}(Z / \Gamma_Z) +
       \sum_{Z \in \Sigma'} c' \ncv_{\geq \eps_2'}(Z / \Gamma_Z)
   \leq (c' + 1) \sum_{Z \in \Sigma'} \ncv_{\geq \eps_2'}(Z / \Gamma_Z)
\end{split}
\end{equation*}
(here we assume $\eps_2' < \eps_1$). Our problem is thus reduced to the case of maximal stable submanifolds.

We will henceforth assume that $\Sigma$ is a set of maximal stable proper submanifolds, and prove that there exist $\eps_2$ and $c$ for which \eqref{eq:bound-for-stable-ess-vol} holds. We may assume that $\eps_1 < \eps$.

\medskip
Let $W \isom W_1 \times \R^d$ be the Clifford splitting of $W$ with respect to $\Gamma$.
Suppose $Y$ is a maximal stable submanifold, and $A_Y$ is the set of stable elements such that $\min d_\alpha < \eps$ for every $\alpha \in A_Y$, and $\bigcap_{\alpha \in A_Y} \Min(\alpha) = Y$. By maximality, $\Min \alpha = Y$ for every $\alpha \in A_Y$.

By Lemma \ref{lem:stable-clifford-commute}, every element in $A_Y$ commutes with the Clifford transformations in $\Gamma$. Hence, every Clifford isometry of $W$ keeps $Y$ invariant, and therefore restricts to a Clifford isometry of $Y$. Therefore, the Clifford splitting of $Y$ (with respect to $\Gamma_Y$) can be made compatible with that of $W$\!; it can be written as $Y = Y_1 \times \R^k \times \R^d$, with $Y_1 \times \R^k \subset W_1$.
Since we are only interested in the projections of isometries to $W_1$ and to $Y_1$, we may simply assume $d=0$.

Every $\beta \in \Gamma_Y$ can be written in the form $\beta = (\beta_1, \beta_2)$ with $\beta_1 \in \Isom(Y_1)$ and $\beta_2 \in \Isom(\R^k)$. Denote by $p_1$ the projection of $\Gamma_Y$ to $\Isom(Y_1)$. Recall that $p_1(\Gamma_Y)$  acts properly on $Y_1$.

We will now describe a procedure that associates to a point in the $\eps_1$-thick part of $Y_1 / p_1 (\Gamma_{Y})$ a point in the $\eps_2$-thick part of $W / \Gamma$ ($\eps_2$ will be defined in this process).

Let $y' \in Y_1$ be a point with $d_{p_1(\Gamma_Y)} (y') \geq \eps_1$. Let $y = (y', 0) \in Y_1 \times \R^k = Y$.
Let $c : [0,\infty) \to W$ be a geodesic ray with $c(0) = y$ and $c'(0) \perp Y$. For every $\alpha \in A_Y$, $\alpha$ keeps $Y$ invariant, and $d_\alpha \circ c$ has a unique minimum at $0$. Hence, by convexity $d_\alpha \circ c$ is monotone increasing and unbounded. Fix $z = c(t_0)$, $t_0 > 0$ such that $d_\alpha(z) \geq \eps / 2$ for all $\alpha \in A_Y$, and $d_\alpha(z) = \eps/2$ for some $\alpha \in A_Y$. We henceforth fix $\alpha$ to be the latter.

We proceed to find a constant $\eps_4$ for which $\Gamma_{\eps_4}(z)$ is finite and bounded. Fix $\eps_3 = \eps_1 / 2 m M$.
Suppose that $d_\beta(z) < \eps_3$ for some $\beta \in \Gamma$. By Lemma \ref{lem:stablizing}, $\beta^j$ is $m$-stable for some $j < M$. Since $d_{\beta^j}(y) < \eps$, $\beta^j$ and $\alpha$ commute. It follows that $\beta^j$ keeps $Y = \Min(\alpha)$ invariant. Since $y$ is the projection of $z$ to $Y$, $d_{\beta^j}(y) \leq d_{\beta^j}(z) \leq j d_\beta(z) < M \eps_3 < \eps_1$.

Let us turn our attention to the projections of $\beta^j$. On one hand, $d_{\beta_1^j}(y') < \eps_1$, so by our hypotheses $\beta_1^j$ must fix $Y_1$ pointwise. Therefore, $\beta^j$ is in the kernel of the projection $\Gamma \to \Gamma_2$. Recall that the action of this kernel on $Y_2$ is proper and cocompact. Hence, by the Bieberbach theorem \cite{thurston-book}*{\S 4}, $\beta_2^{jk}$ is either trivial or a translation for some $k \leq m$. Therefore, $\beta^{jk}$ is either trivial, or acts as a Clifford translation on $Y$. The latter is impossible; indeed, it would imply that $\Min(\beta^{jk}) \supseteq Y$, and by maximality of $Y$, $\Min(\beta^{jk})=Y$ (recall that we are assuming there are no Clifford translations on $W$). But this would imply that $\beta^{jk} \in A_Y$ and contradicts the choice of $z$, because $d_{\beta_{jk}}(z) < m M \eps_3 < \eps / 2$. Thus, the order of $\beta$ is bounded by $m M$.

Now let $\eps_4 = \eps_3 / 2m$. Let $N$ be the normal nilpotent subgroup of $\Gamma_{\eps_4}(z)$. Then $N$ is generated by elements that translate $z$ by less than $2m \eps_4 = \eps_3$. It follows that $N$ is generated by elliptic elements whose order is bounded by $m M$. Therefore, $N$ is abelian, and its exponent is bounded by $(m M)!$. Since $N$ is an abelian subgroup of $\SO(n)$, it can be generated by a set of $n$ elements, hence $\abs{\Gamma_{\eps_4}(z)} \leq m (m M!)^n$.

Now, by Lemma \ref{lem:inj-point-near-small-group} there is a positive constant $\eps_2 < \eps_4$ such that if $\abs{\Gamma_{\eps_4}(z)} \leq m (m M!)^n$ then there exists a point $w \in W$ with $d(z,w) < \eps_4 / 4$ such that $d_\Gamma(w) > \eps_2$. The point $w$ is the point associated to $y'$. Note that
\begin{equation}\label{eq:d-alpha-w}
 d_\alpha(w) \leq d_\alpha(z) + 2 \cdot \frac{\eps_4}{4} < \eps.
\end{equation}

\medskip
Now let $\Sigma = \{ Y^1, \ldots, Y^s \}$ be a set of maximal stable proper submanifolds of $W$. For each $i$, choose a maximal set of points $\{y'_{i,1}, \ldots, y'_{i,t_i} \}$ in $Y^i$ that projects to an $\eps_1$-discrete set in the $\eps_1$-thick part of $Y^i_1 / p_1(\Gamma_{Y^i})$.
To each point $y'_{i,j}$ we attach a point $w_{i,j} \in W$ by the procedure described above. We have seen that these points project to the $\eps_2$-thick part of $W / \Gamma$. Let us denote the projection of $w_{i,j}$ by $\bar w_{i,j}$. The set of $\bar w_{i,j}$'s is not necessarily $\eps_2$-discrete. However, we claim that for every point $\bar w_{i,j}$, there are at most $2^n - 1$ points $\bar w_{k,l}$ for which $d(\bar w_{i,j}, \bar w_{k,l}) < \eps_2$. If this is true, we can find in $\{ \bar w_{i,j} \}$ a subset of size $\lfloor \abs{\{ \bar w_{i,j} \}} / 2^n \rfloor$ which is $\eps_2$-discrete. Using the volume estimates of Proposition \ref{prop:discrete-vol-est}, this will complete the proof. We turn to prove this final claim.

First, we show that if that if $d( \bar w_{i,j}, \bar w_{i,k} ) < 2 \eps_2$ then $j=k$. Indeed, if this inequality holds, we may assume the points are chosen such that $d(w_{i,j},w_{i,k}) < 2 \eps_2$. By the construction, it follows that
\begin{equation*}
  d(y'_{i,1},y'_{i,2}) \leq d(y_{i,1},y_{i,2}) < 2 \eps_2 + 2 \cdot \frac{\eps_4}{4} < \eps_1,
\end{equation*}
which is only possible if $j=k$.

Next, suppose there is a point $\bar w_{i,j}$ for which there are $2^n$ other points of distance $< \eps_2$. Then each two such points are at distance $< 2\eps_2$. By the preceding paragraph, the points correspond to different submanifolds, and we may thus renumber the indices so that $d(\bar w_{1,1}, \bar w_{i,1}) < \eps_2$ for $i=2,\dots,2^n + 1$. Also, we may assume that $d(w_{1,1}, w_{i,1}) < \eps_2$.

By our construction and \eqref{eq:d-alpha-w}, there are stable elements $\alpha_i$ ($1 \leq i \leq 2^n+1$) such that $d_{\alpha_i}(w_{i,1}) < \eps$. Hence, by their stability, these elements commute, and the submanifolds $Y_i = \Min(\alpha_i)$ have a non-empty intersection. Moreover, at the point of intersection, the submanifolds intersect orthogonally (see proof in \cite{bgs-book} for additional details). There is a one-to-one correspondence between these subspaces and their tangent space at the point of intersection, and the tangent subspaces are orthogonal, as well. Hence, there is an orthogonal basis of the tangent space, such that the tangent space to each of the $Y_i$'s is spanned by vectors of this basis. Hence, there are at most $2^n$ such spaces (= the number of subsets of the basis). This contradicts our assumption that $Y_1,\ldots,Y_{n+1}$ are distinct subspaces.
\end{proof}

\section{Proof of Theorems \ref{thm:finite-subgroups} and \ref{thm:finite-subgroups-high-rank}}\label{sec:proofs}

First, we make a reduction of Theorem \ref{thm:finite-subgroups} to the case $G$ has trivial center. Suppose the theorem holds for groups without center, and let $G$ be a connected semisimple Lie group with finite center $Z$. Denote by $\pi$ the projection of $G$ to the adjoint group $G' = G / Z$.

Let $\Gamma < G$ be a lattice. Since $Z$ is finite, every maximal finite subgroup in $\Gamma$ contains $\Gamma \cap Z$. It is easy to check that this implies that $f(\Gamma) = f(\pi(\Gamma))$. Therefore,
\begin{equation*}
 f(\Gamma) = f(\pi(\Gamma)) \leq c \cdot \vol(G' / \pi(\Gamma)) \leq c \cdot \vol(G / \Gamma),
\end{equation*}
where $c = c(G')$.

Assume now that $G$ has trivial center. Let $K$ be a maximal compact subgroup. The associated symmetric space $X = K \backslash G$ is of non-compact type, and we may assume the metric is normalized to meet the curvature bounds. There exists a constant $D=D(G)$ such that $\vol(X / \Gamma) = D \cdot \vol(G / \Gamma)$. Since $X$ does not have Euclidean factors, $\ncv_{\geq \mu}(X/\Gamma) \leq \vol(X/\Gamma)$. Therefore, Theorem \ref{thm:finite-subgroups} follows directly from the following.

\begin{thm}\label{thm:finite-subgroups-geometric}
  Let $X$ be a global symmetric space of non-positive curvature. For a group $\Gamma$ acting properly by isometries on $X$, denote by $f(\Gamma)$ the number of conjugacy classes of maximal finite subgroups of $\Gamma$. Then there exist constants $\mu = \mu(X)$ and $c = c(X)$ such that $f(\Gamma) \leq c \cdot \ncv_{\geq \mu}(X/\Gamma)$.
\end{thm}
\begin{proof}
  Throughout the proof, $\eps$ and $m$ are constants of the Margulis Lemma, and $M$ is the constant of Lemma \ref{lem:stablizing}. We choose all constants so that these lemmas hold for any symmetric space with dimension $\leq \dim(X)$.

  We will prove by induction that for every $d \leq \dim(X)$ there are constants $\mu_d, c_d$ such that if $Y$ a connected complete totally geodesic submanifold of $X$ of dimension $\leq d$, and $\Gamma$ is a group acting properly on $Y$ then $f(\Gamma) \leq c_d \cdot \ncv_{\geq \mu_d}(Y / \Gamma)$.

  If $\dim(Y)=0$, and $\Gamma$ acts properly on $Y$, then $\Gamma$ is finite, and therefore $f(\Gamma)=1$. We may fix $c_0=1$ and $\mu_0$ to be any positive number.

  Suppose now that $\mu_i$ exist for $i=0,\ldots,d-1$, and let us prove that $\mu_d$ exists. Let $Y$ be a connected complete totally geodesic submanifold of $X$ of dimension $d$, and let $\Gamma$ be a group of isometries acting properly on $Y$.

  \noindent {\bf Step 1.}
  Let us first assume that $\Gamma$ contains Clifford isometries. Let $T \leq \Gamma$ be the subgroup of Clifford isometries, and let $Y = Y_1 \times Y_2$, $Y_2 \isom \R^k$ ($k \geq 1$) be the Clifford splitting with respect to $\Gamma$.
  Denote by $\Gamma_i$ the projection of $\Gamma$ to $\Isom(Y_i)$, $i=1,2$. Recall that $\Gamma_1$ acts properly on $Y_1$, and that the subgroup $T$ is contained in the kernel of the projection $\Gamma \to \Gamma_1$, and acts cocompactly on $Y_2$.

  Our goal is to show that a maximal finite subgroup of $\Gamma$ has bounded number of possible projections to $\Gamma_1$ and $\Gamma_2$, up to conjugation in $\Gamma$. Then we will show this yields the appropriate bound on the number of conjugacy classes of maximal finite subgroups in $\Gamma$.

  For a finite subgroup $H_1$ in $\Gamma_1$, define $\Gamma(H_1) = p_1^{-1}(H_1)$ and $\Gamma_2(H_1) = p_2(\Gamma(H_1))$. Since the restriction of $p_2$ to $\Gamma(H_1)$ has finite kernel, $\Gamma_2(H_1)$ acts properly on $Y_2$. Moreover, $T$ is contained in $\Gamma(H_1)$ and injects into $\Gamma_2(H_1)$, and therefore $\Gamma_2(H_1)$ acts cocompactly on $Y_2$.

  We claim that $f(\Gamma_2(H_1))$ is bounded by a constant $D$ that depends only on $\dim(X)$. This follows immediately from the fact that, up to isometry, there are finitely many crystallographic groups of every dimension \cite{auslander-crystallographic}.

  Now let $\{ H^{(1)}_1, \ldots, H^{(1)}_r \}$ be a maximal collection of non-conjugate maximal finite subgroups in $\Gamma_1$. For every $1 \leq i \leq r$, let $\{ H^{(2)}_{i,1}, \ldots, H^{(2)}_{i,s_i} \}$ be a maximal collection of non-conjugate maximal finite subgroups of $\Gamma_2(H^{(1)}_i)$.
  By the induction hypothesis, $r = f(\Gamma_1) \leq \ncv_{\geq \mu_{d-1}}(Y_1 / \Gamma_1) = \ncv_{\geq \mu_{d-1}}(Y / \Gamma)$ (the latter equation being the definition of $\ncv$). Also, $s_i \leq D$ for every $i$.

  Let $H < \Gamma$ be a maximal finite subgroup. By conjugating in $\Gamma$ we may assume that the projection of $H$ to $\Gamma_1$ is contained in $H^{(1)}_i$ for some $i$. Furthermore, by conjugating by an element of $\Gamma(H^{(1)}_i)$ we may assume that $H$ projects into $H^{(2)}_{i,j}$ for some $1 \leq j \leq s_i$. Note that this conjugation does not change the projection of $H$ to $\Gamma_1$. Our claim will be proved if we show that there is a single maximal finite subgroup with such projections.

  To this end, let $y_1 \in Y_1, y_2 \in Y_2$ be points fixed by $H^{(1)}_i,H^{(2)}_{i,j}$, respectively. If $K_1,K_2$ are two maximal finite subgroups of $\Gamma$ which both project to $H^{(1)}_i, H^{(2)}_{i,j}$, then they both fix $(y_1,y_2)$. But by maximality, they are both equal to the stabilizer of this point in $\Gamma$.

  We conclude that
  \begin{equation}\label{eq:step1}
    f(\Gamma) \leq D \ncv_{\geq \mu_{d-1}}(Y / \Gamma).
  \end{equation}

  \smallskip
  \noindent {\bf Step 2.}
  We henceforth assume $\Gamma$ does not contain Clifford isometries.

  Let $\Sigma$ be a maximal set of non-conjugate stable singular proper submanifolds of $Y$. For $Z \in \Sigma$, denote by $\Gamma_Z$ the stabilizer of $Z$ in $\Gamma$, by $\Gamma_Z^1$ the fixator (pointwise stabilizer) of $Z$ in $\Gamma$, and by $f_Z$ the number of non-conjugate maximal finite subgroups of $\Gamma$ that are contained in $\Gamma_Z$.

  Observe that if $H$ is a maximal finite subgroup of $\Gamma$ that keeps $Z$ invariant, then it fixes some point $z \in Z$. By maximality, $H = \Gamma_z$, and therefore $H$ contains $\Gamma_Z^1$.
  Therefore, maximal finite subgroups that are contained in $\Gamma_Z$ are in a natural one-to-one correspondence with maximal finite subgroups of $\Gamma_Z / \Gamma_Z^1$. Moreover, this correspondence respects conjugation by elements of $\Gamma_Z$.

  It follows that $f_Z \leq f(\Gamma_Z / \Gamma_Z^1) \leq \ncv_{\geq \mu_{d-1}}(Z / \Gamma_Z)$, by the induction hypothesis.
  By Proposition \ref{prop:inductive-vol-bound}, we have
  \begin{equation}\label{eq:step2}
    \sum_{Z \in \Sigma} f_Z \leq \sum_{Z \in \Sigma} \ncv_{\geq \mu_{d-1}}(Z / \Gamma_Z) \leq E_1 \cdot \ncv_{\geq \delta_1}(Z / \Gamma_Z),
  \end{equation}
  with $\delta_1 = \delta_1(\mu_{d-1})$, $E_1 = E_1(\mu_{d-1})$.

  It remains to bound the number of non-conjugate maximal finite subgroups that do not stabilize any stable singular proper submanifold.

  \smallskip
  \noindent {\bf Step 3.}
  Let $\calH$ be a set of non-conjugate maximal finite subgroup of $\Gamma$ that do not stabilize any stable singular proper submanifold.
  Let $H \in \calH$, and let $y \in \Fix(H)$. By maximality, $H = \Gamma_y$.

  We claim that there is no hyperbolic element $\alpha \in \Gamma$ with $\Min(\alpha) \supseteq \Fix(H)$ and $d_\alpha(y) < \eps / M$ (recall, $M$ is the constant of Lemma \ref{lem:stablizing}). Suppose $\alpha$ is such an element. Then for some $j \leq M$, $\alpha^j$ is stable, and $d_{\alpha^j}(y) < \eps$. The elements in $\{ \gamma \alpha^j \gamma^{-1} \; | \; \gamma \in H\}$ are all stable and contained in $\Gamma_\eps(y)$, hence they commute.
  It follows that $Z = \bigcap_{\gamma \in H} \Min(\gamma \alpha^j \gamma^{-1})$ is non-empty, and it is a $H$-invariant stable singular submanifold. By our assumption, there are no Clifford isometries in $\Gamma$, so $Z \neq Y$. But this contradicts our assumption.

  Now, by Proposition \ref{prop:mini-km} we may replace $y$ with a point $y \in \Fix(H)$ such that $\Gamma_{\eps / 2mM}(y) = H$.

  Following the same reasoning, every element $\alpha \in H$ has order bounded by $M$; otherwise, $\bigcap_{\gamma \in H} \Min(\gamma \alpha^j \gamma^{-1})$ would be a non-empty $H$-invariant stable singular proper submanifold, for some $j \leq M$. By identifying $H$ with a subgroup of $\SO(n)$, we see that its normal abelian subgroup is generated by at most $n$ elements. Since the index of this subgroup is a most $m$, we deduce that $\abs{H} \leq M^n m$.

  By Lemma \ref{lem:inj-point-near-small-group} there is a positive constant $\delta_2 < \eps/8$ such that if \begin{equation*}
   \abs{\Gamma_{\eps / 2 m M}(y)} \leq M^n m
  \end{equation*}
  then there is a point $z \in Y$ with $d(x,y) < \eps / 8$ such that $d_\Gamma(y) > \delta_2$. We denote this point $z_H$.

  We claim that the set of points $\{ z_H : H \in \calH \}$ projects to a $\delta_2$-discrete set in the $\delta_2$-thick part of $Y / \Gamma$. By the choice of $z_H$, the projection is indeed in the $\delta_2$-thick part. It remains to show discreteness. To this end, let  $H_1,H_2 \in \calH$, let $y_1, y_2$ be the corresponding fixed points, and let $z_1,z_2$ be the points obtained for each group. Suppose, in contradiction, that the projections of $z_1,z_2$ are not $\delta_2$-separated. By replacing $H_2$ by a conjugate, we may assume  $d(z_1,z_2) < \delta_2$. This implies that
  \begin{equation*}
    d(y_1,y_2) \leq \delta_2 + 2 \cdot \eps / 8 < \eps/2.
  \end{equation*}
  Consequentially, $H_1 = \Gamma_{y_1} \leq \Gamma_\eps(y_2) = H_2$, hence $H_1 = H_2$, a contradiction.
  By the volume estimate of \ref{prop:discrete-vol-est},
  \begin{equation}\label{eq:step3}
    \abs{\calH} \leq E_2 \cdot \vol_{\geq \delta_2/2}(Y / \Gamma) = E_2 \cdot \ncv_{\geq \delta_2/2}(Y / \Gamma),
  \end{equation}
  where $E_2$ is a constant depending on $X$.

  \medskip \noindent
  {\bf Step 4.}
  Take $c_d = \max(D,E_1+E_2)$ and $\mu_d = \min(\delta_1,\delta_2/2,\mu_{d-1})$. It follows easily from \eqref{eq:step1}, \eqref{eq:step2}, and \eqref{eq:step3} that
  \begin{equation*}
    f(\Gamma) \leq c_d \cdot \vol^{nc}_{\geq \mu_d} (Y / \Gamma).
  \end{equation*}
\end{proof}

\begin{proof}[Proof of Theorem \ref{thm:finite-subgroups-high-rank}]
Assume $G$ has $\R\!\hrank$ at least 2, and Kazhdan's property (T). In this case, we have the following result
\begin{thm}
Let $\Gamma_n$ be a sequence of irreducible lattices in $G$ such that $\vol(G / \Gamma_n) \to \infty$. Then for every $R > 0$,
\begin{equation*}\label{eq:thick-vol-sublinear}
  \lim_n \frac {\vol_{\leq R} (X / \Gamma_n)} {\vol(X/\Gamma_n)} = 0.
\end{equation*}
\end{thm}
This is Corollary 4.10 of \cite{samurai-l2} (a version for simple groups appears as Corollary 2.5 of \cite{samurai-announce}).

We now reflect on the proof of theorem \ref{thm:finite-subgroups-geometric}. The basic technique of the proof is, essentially, to bound the size of certain sets by assigning to each element of the set a point in the $\delta$-thick part of $X / \Gamma$ (for some $\delta > 0$ which is smaller than $\eps$), and making sure these points are $\delta$-discrete. This is done in two places.
The first, is in step 2 of the proof, and relies on Proposition \ref{prop:inductive-vol-bound}. The second is done directly in step 3. (Note that we are restricting our attention to the proof to the last step of the induction, i.e. $d = \dim(X)$). If we look closer into the proofs of the Proposition and of step 3, we see that in both cases the points in $\delta$-thick part are contained in the $\frac{\eps}{2}$-thin part of $X / \Gamma$. In fact, this is crucial to the proof of $\delta$-discreteness. Therefore, we also have a bound
\begin{equation*}
f(\Gamma) \leq c' \cdot \vol_{\leq \eps}(X / \Gamma)
\end{equation*}
(of course, this bound does not lend itself to induction, and is therefore not highlighted in the proof of the theorem). The proof is completed by appealing to Theorem \ref{eq:thick-vol-sublinear}.
\end{proof}

\section{A construction in \texorpdfstring{$\SO(d,1)$}{SO(d,1)}}\label{sec:so}

\subsection{}
We first describe a general setting that assures that a group contains a sequence of subgroups of finite index for which the number of non-conjugate maximal finite subgroups grows linearly with index. Later, we will construct lattices in $\SO(d,1)^\circ$ which realize this setting.

\begin{prop}\label{prop:linear-growth}
  Let $\Phi$ be a countable group with a finite subgroup $H$. Assume there is an epimorphism $\varphi : \Phi \to \Z$ such that the normalizer of $H$ is contained in $\ker \varphi$.
  Let $\Phi_n$ be the kernel of the composition $$\Phi \xra{\varphi} \Z \to \Z / n \Z.$$
  Then there are $n$ non-conjugate subgroups of $\Phi_n$ that are all conjugate to $H$ in $\Phi$.
\end{prop}
\begin{proof}
Fix some $t \in \varphi^{-1}(1)$.
Consider the subgroups $H_i = t^i H t^{-i}$, $i \geq 0$. Since they are finite, they are all contained in $\ker \varphi$, hence in $\Phi_n$, for all $n$.

We claim that $H_0,\ldots,H_{n-1}$ are non-conjugate in $\Phi_n$. Indeed, if $H_i$,$H_j$ ($0\leq i,j < n$) are conjugate then $t^{-j} g t^i$ normalizes $H$ for some $g \in \Phi_n$. By our assumption on $H$ and the definition of $\Phi_n$, $\varphi(g) = j-i \equiv 0 \pmod n$, and therefore $i=j$.
\end{proof}

In our application, $\Phi$ will be a lattice. As the covolume of $\Phi_n$ is proportional to $n = [\Phi : \Phi_n]$, the number of non-conjugate finite subgroups constructed in the previous proposition grows linearly with the covolume of $\Phi_n$. We also remark that if we start with a \emph{maximal} finite subgroup $H$, the construction yields non-conjugate maximal subgroups.

\subsection{}
Now, we construct an appropriate lattice in $G = \SO(d,1)^\circ$, the connected component of $\SO(d,1)$.
The starting point of our construction is the result of Millson \cite{millson-betti} concerning the positivity of the first Betti number of certain congruence subgroups in standard arithmetic lattices of $G$.

Let $X = \{ (x_1,\ldots,x_{d+1}) \; | \; x_1^2 + \cdots + x_d^2 - x_{d+1}^2 = -1, x_{d+1} > 0 \}$
be the upper sheet of the hyperboloid, induced with a metric from the Minkowski space. This is a model for the $d$-dimensional real hyperbolic space. The linear action of $G$ on $\R^{d+1}$ induces an isometric action on $X$.

Let $q$ be a prime, and let $Q(X_1,X_2,\ldots,X_{d+1}) = X_1^2 + \cdots + X_d^2 - \sqrt q X_{d+1}^2$ be a quadratic form. Let $\calO$ be the ring of integers in $\Q[\sqrt q]$. Clearly, $\SO(Q, \R)$ is conjugate to $\SO(d,1)$, and we denote by $\Gamma$ the corresponding conjugate of $\SO(Q, \calO)$ in $\SO(d,1)$, intersected with $\SO(d,1)^\circ$. Then $\Gamma$ is a lattice in $G$. For an ideal $\p$ in $\calO$, let $\Gamma(\p) = \{ \gamma \in \Gamma \; | \; \gamma \equiv I \mod \p \}$. We will fix a certain $\p$ shortly, and denote $\Delta = \Gamma(\p)$.

Denote by $Y$ the hyperplane $\{ (x_1,\ldots,x_{d+1}) \in X \; | \; x_1 = 0 \}$.
Let $\Delta_Y$ the subgroup of isometries in $\Delta$ stabilizing $Y$ (in Millson's work, this is described as the group of isometries commuting with the reflection around $Y$). For almost all $\p$, $\Delta$ is torsion-free, and it is shown that for $\p$ deep enough, $Y / \Delta_Y$ embeds as a non-separating oriented hyperplane in $X / \Delta$. We fix such ideal $\p$. Note that $\Delta$ is normal in $\Gamma$.

As explained in \cite{millson-betti}, the homology class carried by $Y / \Delta_Y$ is non-trivial, and its Poincar\'e dual yields an epimorphism $\phi : \Delta \to \Z$.

Let $H < \Gamma$ be the group of isometries stabilizing $Y$ and fixing the point $(0,\ldots,0,1) \in Y$. Note that $H$ is not trivial; it contains, for example, the diagonal matrices with -1 in two of the first $d$ entries, and 1 on the rest. Moreover, by considering all the matrices of this form, we see that $(0,\ldots,0,1)$ is the \emph{unique} fixed point of $H$.

Since $H$ normalizes $\Delta$, it acts by isometries on $X / \Delta$. Furthermore, since $H$ stabilizes $Y$ it normalizes $\Delta_Y$ and stabilizes $Y / \Delta_Y$. It follows that $H$ fixes the homology class of $Y / \Delta_Y$ and its Poincar\'e dual. In other words, $\phi(h \gamma h^{-1}) = \phi(\gamma)$ for all $\gamma \in \Delta$ and $h \in H$. By the following lemma, whose proof is left to the reader, we can extend $\phi$ to $H \ltimes \Delta$.

\begin{lemma}
  Let $\Delta < \Gamma$ be a subgroup, and $H < \Gamma$ a finite subgroup contained in the normalizer of $\Delta$. A homomorphism $\phi : \Delta \to \Z$ can be extended to $H \Delta$ iff $\phi(h \gamma h^{-1}) = \phi(\gamma)$ for all $\gamma \in \Delta$ and $h \in H$.
\end{lemma}

We will denote the extension of $\phi$ to $H \ltimes \Delta$ by $\tilde \phi$.
To use Proposition \ref{prop:linear-growth} (with $H \ltimes \Delta$ as the group $\Phi$), we must show that the normalizer of $H$ in $H \ltimes \Delta$ is contained in $\ker \tilde \phi$. Since $H \leq \ker \tilde \phi$, it suffices to show that $H$ is self-normalizing in $H \ltimes \Delta$. To this end, suppose $\gamma \in \Delta$ normalizes $H$. Then $\gamma$ must fix the unique fixed point of $H$, but since $\Delta$ is torsion-free this forces $\gamma$ to be $1$.

Finally, observe that $H$ is a maximal finite subgroup in $H \ltimes \Delta$. Indeed, if $g \in H \ltimes \Delta$, $g \not\in H$, then there exists $1 \neq \gamma \in \Delta$ such that $g \gamma^{-1} \in H$. Therefore, the group generated by $H$ and $g$ contains $\gamma$, and is therefore infinite.

\section{Isotropy subgroups}\label{sec:isotropy}

Theorem \ref{thm:finite-subgroups} does not provide a bound on the number of conjugacy classes of non-maximal finite subgroups (or torsion elements) in $\Gamma$, because we have no effective bound on the size of finite subgroups of $\Gamma$.
However, we will now show that there is an effective bound on the number of non-conjugate \emph{isotropy} subgroups in $\Gamma$, i.e. subgroups that are stabilizers of points in $X$. It is easy to see that a stabilizer of submanifold $Y$ of $X$ is a stabilizer of some (but not every) point in $Y$. Hence there is a one-to-one correspondence between isotropy subgroups and fixed submanifolds.

Let us begin by proving a lemma regarding linear groups.

\begin{define}
Let $T$ be a subset of $\GL_n(\C)$. A set of vectors $\calB$ in $\C^n$ \emph{splits} $T$, if $\calB$ contains a basis of eigenvectors for every $t \in T$.
\end{define}

If $H < GL_n(\C)$ is an abelian group consisting of semisimple elements then there is a set vectors of size $n$ that splits $H$. This is a restatement of the fact that the elements of $H$ can be simultaneously diagonalized. If $H$ is finite but not abelian, there is generally no uniform bound (depending only on $n$) on the size of a splitting set. Indeed, the minimal splitting set for a dihedral group $D_{2n}$ realized in the standard way as a subgroup of $GL_2(\C)$ is $2n$. Nonetheless, in this example, there is a cyclic subgroup of index 2, and all other elements (of order 2) constitute one or two conjugacy classes. Thus, if one considers splitting ``up-to-conjugacy'', the size of a minimal splitting set is $4$ or $6$. We make generalize this phenomenon:

\begin{lemma}
  There exists a function $f(n)$ such that for every finite group $G$ of $\GL_n(\C)$ there exists a set $\calB$ such that $G \calB$ splits $G$ and $\abs{\calB} < f(n)$.
\end{lemma}

\begin{remark}
  Saying that $G \calB$ splits $G$ is equivalent to saying that $\calB$ splits a set of representatives of the conjugacy classes in $G$.
\end{remark}

\begin{proof}
  By a theorem of Jordan \cite{curtis-reiner-book}*{36.13}, there is a normal abelian subgroup $A \lhd G$ such that $[G:A] < j(n)$, where $j$ is a function of $n$ alone. It therefore suffices to prove that for any coset, $gA$, there is a set of vectors $\calC$ such that $G \calC$ splits $gA$, and $\abs{\calC}$ is bounded by a constant that depends only on $n$.

  Let $gA$ be a coset of $A$ in $G$. Set $H = \ang{A,g}$. Assume first that $H$ acts irreducibly on $\C^n$. Since $A$ is abelian, $C := C_A(g) \leq Z(H)$, and by Schur's lemma, $Z(H) < Z(\GL_n(\C))$.

  Let
  \begin{equation*}
  K=[A,g^{-1}] = \{ [a,g^{-1}] \; : \; a\in A \}.
  \end{equation*}
  It is straightforward to check that $K$ is a subgroup of $A$, because $A$ is normal and abelian. Moreover, the map $a \mapsto [a,g^{-1}]$ is a homomorphism of $A$ onto $K$ with kernel $C$. Note that $K \leq \SL_n(\C)$, and hence $K \cap C \leq Z(\SL_n(\C))$. Thus, $\abs {K \cap C} \leq n$. We conclude that
  \begin{equation*}
    \abs{KC} = \frac{\abs{K} \abs{C}}{\abs{K \cap C}} \geq \frac{\abs{K}\abs{C}}{n} =
    \frac{\abs{A}}{n},
  \end{equation*}
  hence the index of $KC$ in $A$ is at most $n$.

  Choose representatives $a_1,\ldots,a_s$ ($s \leq n$) for the cosets of $KC$ in $A$. For each $i$, choose a basis of eigenvectors for $g a_i$, and let $\calC$ be the union of these bases. Clearly, $\abs{\calC} \leq n^2$. We claim that $H \calC$, and moreover $A \calC$, splits $gA$. To this end, let $y \in gA$. For some $i$, $y \in g a_i KC$ and write $y = g a_i [a,g^{-1}] c$ ($[a,g^{-1}] \in K$, $c \in C$). Then
  \begin{equation*}
    y = c g a_i a (g^{-1} a^{-1} g) = c g (g^{-1} a^{-1} g) a_i a = c a^{-1} g a_i a.
  \end{equation*}
  Since $\calC$ splits $g a_i$, $A \calC$ splits $a^{-1} g a_i a$, and since $c$ is a scalar matrix, it also splits $y$.

  In the general case, decompose $\C^n$ to $H$-irreducible subspaces $V_1,\ldots,V_r$. In each subspace $V_i$, we take a set $\calC_i$ such that $A \calC_i$ splits the restriction of $gA$ to $V_i$, and that $\abs {\calC_i} \leq \dim(V_i)^2$. Now $\calC = \bigcup \calC_i$ has the required properties: $A\calC$ splits $gA$ and $\abs{\calC} \leq \sum \dim(V_i)^2 \leq n^2$.
\end{proof}

\begin{remark}
  The proof shows in fact that we may take a set $\calB$ such that $\abs{\calB} < f(n)$ and $A \calB$ already (rather than $G \calB$) is a splitting set.
\end{remark}

\begin{corl}\label{corl:fixed-subspaces}
  Let $G$ be a finite subgroup of $GL_n(\C)$. Then the number of $G$-orbits in the set of fixed subspaces $\{ Fix(g) : g \in G \}$ is bounded by a function of $n$.
\end{corl}

\begin{corl}\label{corl:isotropy-sg-bound}
  Let $X$ be a global symmetric space of non-compact type, and $\Gamma$ a discrete subgroup of $\Isom(X)$. Denote by $i(\Gamma)$ the number of conjugacy classes of isotropy subgroups in $\Gamma$.
  There is a constant $c' = c'(X)$ such that
  \begin{equation*}
    i(\Gamma) \leq c' \cdot f(\Gamma).
  \end{equation*}
\end{corl}
\begin{proof}
  The claim will follow if we show that there is a uniform bound on the number of isotropy subgroups contained in a maximal finite subgroup of $\Gamma$.

  Let $H < \Gamma$ be a maximal finite subgroup, and let $x \in \Fix(H)$. We may identify $H$ as a subgroup of $\SO(n)$ ($n = \dim(X)$) through the action of $H$ on $T_x(X)$, the tangent bundle at $x$. We note that by replacing the action of $H$ on $X$ by an action on $T_x(X)$, the set of isotropy subgroups in $H$ is unchanged. Since conjugacy classes of isotropy subgroups correspond to $H$-orbits of fixed subspaces, the assertion follows at once from the previous corollary.
\end{proof}

\begin{corl}[Theorem \ref{thm:strata} of the introduction]
     Let $X$ be a symmetric space $X$ of non-compact type. Let $M = X / \Gamma$ be an $X$-orbifold. Denote by $s(M)$ the number of strata in the natural orbifold stratification of $M$. Then
   \begin{equation*}
     s(M) \leq c \cdot \vol(M),
   \end{equation*}
   with a constant $c = c(X)$. If the rank of $X$ is at least 2, and $\Isom(X)$ has property (T), then for any sequence $M_n$ of irreducible $X$-orbifolds that are pairwise non-isometric we have
   \begin{equation*}
     \lim_{n \to \infty} \frac{s(M_n)}{\vol(M_n)} = 0.
   \end{equation*}
\end{corl}
\begin{proof}
  Strata correspond to conjugacy classes of isotropy subgroups. Hence, this follows immediately from Corollary \ref{corl:isotropy-sg-bound}, and Theorems \ref{thm:finite-subgroups} and \ref{thm:finite-subgroups-high-rank}.
\end{proof}

\end{document}